\documentclass[11pt,reqno]{amsart}
\usepackage{mathrsfs}
\usepackage{}

\def\subjclass#1{{\renewcommand{\thefootnote}{}%
\footnote{\emph{Mathematics Subject Classification (2010):} #1}}}

\usepackage{amsfonts}
\usepackage{amssymb}

\DeclareMathOperator{\Curl}{Curl}
\DeclareMathOperator{\curl}{curl}
\DeclareMathOperator{\divg}{div}

\date{\today}

\hoffset=- 2cm \voffset=- 0cm 
 \theoremstyle{plain}
\newtheorem{Thm}{Theorem}

\newtheorem{Rem}[Thm]{Remark}
\newtheorem{Lem}[Thm]{Lemma}

\usepackage{amssymb}
\usepackage{color}

\newcommand {\p}{\partial}
\newcommand{\q}{\quad}

\def\0{\mathbf 0}
\def\a{\alpha}

\def\d{\nabla}

\def\lam{\lambda}

\def\O{\Omega}

\def\R{\mathbb R}

\def\h{\mathbf h}
\def\u{\mathbf u}

\errorcontextlines=0 \numberwithin{equation}{section}
\numberwithin{Thm}{section}
\allowdisplaybreaks

\begin{document}
\large

\title[Applications of a formula on Beltrami flow]{Applications of a formula on Beltrami flow}

\author[]{Yong Zeng}
\author[]{Zhibing Zhang}

\address{Yong Zeng: School of Mathematics and Statistics, Chongqing Technology and Business University, Chongqing 400067, PR China; }
\email{zengyong0702@126.com}

\address{Zhibing Zhang: School of Mathematics and Physics, Anhui University of Technology, Ma'anshan 243032, PR China; }
\email{zhibingzhang29@126.com}

\thanks{}



\begin{abstract}
In this note, we obtain uniqueness results for Beltrami flow in both bounded and unbounded domain with nonempty boundary by establishing an elementary but useful formula involving operators $\divg$ and $\curl$. We also use this formula to deal with Maxwell and Stokes eigenvalue problems.

\end{abstract}
\maketitle

\section{Introduction}

In this note we study the {\it Beltrami flow}, that is, a vector field $\u$ which satisfies the system
\begin{equation}\label{eq-bel}
  \left\{\begin{aligned}
  &\curl\u\times\u=\0 & \text{in }\O,\\
  &\divg\u=0 &\text{in }\O.
  \end{aligned}\right.
\end{equation}
By establishing an elementary but useful identity, we obtain uniqueness results for Beltrami flow in both bounded and unbounded domain with nonempty boundary. As another interesting application, we use this identity to deal with Maxwell and Stokes eigenvalue problems.

Since $\curl\u\times\u=(\u\cdot\d)\u-\d|\u|^2/2$, each Beltrami flow will then give a special solution to the stationary Euler system. We refer the reader to \cite{Arnold} for the basic properties of Beltrami flows. For some recent results see \cite{EP12,EP15,EP16,EPS16,Pan2016} and references therein. We mention here that the Beltrami flows are also called {\it force-free magnetic fields} in magnetohydrodynamics, since the term $\curl\u\times\u$ models the Lorentz force when $\u$ represents the magnetic field, see \cite{Ch,ChKe,Va} and references therein.

In~\cite{EP15}, the authors constructed  Beltrami fields which satisfy $\curl\u=\lam\u$ in $\R^3$ for nonzero constant $\lam$ and fall off as $|\u(x)|<C|x|^{-1}$ at infinity. In particular, they are in $L^p(\R^3,\R^3)$ for all $p>3$. A similar result can be also found in \cite{LLZ2015}. Recently, N. Nadirashvili \cite{Nadi14} proved a Liouville-type theorem for the globally defined Beltrami flow. He proved that when $\O=\R^3$, a $C^1$ Beltrami flow satisfying either $\u\in L^p(\R^3,\R^3)$, $p\in[2,3]$ or $|\u(x)|=o(|x|^{-1})$ as $|x|\to+\infty$ is in fact trivial, i.e., $\u\equiv \0$ in $\R^3$. In \cite{CC15}, D. Chae and P. Constantin gave a new and elementary proof to a similar result which partially covers the result of N. Nadirashvili. In \cite{ChW2016-2}, D. Chae and J. Wolf succeeded in covering the result of N. Nadirashvili and got some improvements. Concerned with the exterior problem, A. Enciso, D. Poyato and J. Soler \cite{EPS16} considered a related system
$$\left\{\aligned
&\curl\u=f\u &\text{in }\R^3\backslash\overline{\O},\\
&\divg\u=0 &\text{in }\R^3\backslash\overline{\O},
\endaligned\right.
$$
for some  $f$ which is a compactly supported perturbation of a constant proportionality factor $\lambda\in\R\backslash\{0\}$, i.e., $f=\lambda+\varphi$ for some $\varphi\in C^{k,\a}_c(\R^3\backslash\O,\R^3)$, where $\O$ is a $C^{k+1,\a}$ bounded domain homeomorphic to an Euclidean ball. They showed that if $\u\in C^{k+1,\a}(\R^3\backslash\O,\R^3)$ is transverse to $\p\O$ at some point outside the support of $\varphi$, then there exists no $\varepsilon>0$ such that $|\u(x)|=O(|x|^{-1-\varepsilon})$ as $|x|\to +\infty$, otherwise $\u\equiv\0$.

\section{Main Results}
We consider the Beltrami flow with vanishing tangential or vanishing normal boundary condition, and obtain uniqueness results for the Beltrami flow in both bounded and unbounded domain with nonempty boundary. The first result is for the bounded domain.
\begin{Thm}\label{ZZ1}
Let $\Omega\subset\R^3$ be a bounded and star-shaped $C^1$ domain.
Then the system
\begin{equation}
\begin {cases}
\curl\mathbf{u}\times\mathbf{u}=\mathbf{0}& \text{ \rm in } \Omega,\\
\divg\mathbf{u}=0 &\text{ \rm in } \Omega,\\
\mathbf{u}\times\nu=\mathbf{0} &\text{ \rm on } \partial\Omega,\\
\end {cases}
\end{equation}
admits only the trivial solution in $H^1(\O,\R^3)$.
\end{Thm}

\begin{Rem}
$(\rm i)$ S. I. Vainshtein \cite[p.180]{Va} $($see also \cite[p.5638]{CDGT}$)$ showed that the system
\begin{equation}
\begin {cases}
\curl\mathbf{u}=\lambda\mathbf{u}& \text{ \rm in } \Omega,\\
\divg\mathbf{u}=0 &\text{ \rm in } \Omega,\\
\mathbf{u}\times\nu=\mathbf{0} &\text{ \rm on } \partial\Omega,\\
\end {cases}
\end{equation}
admits only the trivial solution, provided that $\O$ is bounded and $\lam$ is a nonzero constant. Our result can be viewed as a generalization to Vainshtein's result since we deal with a more general system, even though our condition on the domain is more restricted.

$(\rm ii)$
If we replace the boundary condition $\mathbf{u}\times\nu=\mathbf{0}$ on $\p\O$ by $\mathbf{u}\cdot\nu=0$ on $\p\O$, then Theorem \ref{ZZ1} does not hold any more. Indeed, let $(\lambda_k,\u_k)$ be the eigen-pairs for the operator $curl$
\begin{equation*}
\begin {cases}
\curl\mathbf{u}_k=\lambda_k\mathbf{u}_k& \text{ \rm in } \Omega,\\
\divg\mathbf{u}_k=0 &\text{ \rm in } \Omega,\\
\mathbf{u}_k\cdot\nu=0 &\text{ \rm on } \partial\Omega,\\
\end {cases}
\end{equation*}
whose existence is ensured by Z. Yoshida and Y. Giga \cite{YG90}, then each $\u_k$ gives a nontrivial solution to the system
\begin{equation*}
\begin {cases}
\curl\mathbf{u}\times\mathbf{u}=\mathbf{0}& \text{ \rm in } \Omega,\\
\divg\mathbf{u}=0 &\text{ \rm in } \Omega,\\
\mathbf{u}\cdot\nu=0 &\text{ \rm on } \partial\Omega.\\
\end {cases}
\end{equation*}

\end{Rem}

The second result is a Liouville-type theorem for the Beltrami flow in unbounded domains.

\begin{Thm}\label{ZZ2}
  Let  $\Omega\subset\R^3$ be an unbounded domain with a $C^1$ boundary and $\R^3\backslash\overline{\Omega}\neq\emptyset$. Assume that $\u\in C^1(\Omega,\R^3)\cap C^0(\overline{\Omega},\R^3)$ is a solution to \eqref{eq-bel} such that $|\u(x)|=o(|x|^{-1})$ as $|x|\to +\infty$. If additionally,
  \begin{itemize}
    \item[(i)]  $\Omega$ is a star-shaped domain and $\u\times\nu=\0$ on $\p\O$, or
    \item[(ii)] $\Omega=\R^3\backslash\overline{D}$, where $D$ is a star-shaped domain, and $\u\cdot\nu=0$ on $\p\O$,
  \end{itemize}
then we must have $\u\equiv\0$ in $\O$.
\end{Thm}

\begin{Rem}
$(\rm i)$
Let $\u$ be a Beltrami flow in $\R^3_+$ with $|\u(x)|=o(|x|^{-1})$ as $|x|\to+\infty$.  Since the half space $\R^3_+=\{x\in\R^3, x_3>0\}$ is  both a star-shaped domain and  the complement of the closure of the star-shaped domain $\R_-^3=\{x\in\R^3, x_3<0\}$, therefore, either $\u\cdot\nu=0$ on $\{x_3=0\}$ or $\u\times\nu=\0$ on $\{x_3=0\}$ will give $\u\equiv \0$ in $\R^3_+$.

$(\rm ii)$ Note that the star-shaped domain $D$ in our Theorem~\ref{ZZ2} can be both bounded and unbounded, thus it contains some domains that can not be covered by the results in \cite{EPS16}.
\end{Rem}

The proofs of Theorem \ref{ZZ1} and \ref{ZZ2} are based on the formula established in Lemma \ref{one-ball}. Another interesting application of this formula is to compare the first Maxwell eigenvalues $\alpha_1,\beta_1$ with the first Stokes eigenvalue $\gamma_1$. The definitions of $\alpha_1,\beta_1,\gamma_1$ can be found in Section 2. For three-dimensional bounded and star-shaped domain, we find that the first Maxwell eigenvalues are strictly smaller than the first Stokes eigenvalue. Moreover, we have

\begin{Thm}\label{ZZ3}
Let~$\O$~be a bounded and star-shaped $C^{1,1}$ domain in $\mathbb{R}^3$. It holds that
\begin{itemize}
    \item[(i)] $\alpha_1=\beta_1<\gamma_1$.
    \item[(ii)] Let $\u$ be a first Maxwell eigenfunction under tangent boundary condition. Then $\u$ satisfies $\u\cdot\nu\not\equiv 0$ and $\curl\u\times\nu\not\equiv\0$ on $\p\O$.
    \item[(iii)] Let $\u$ be a first Maxwell eigenfunction under normal boundary condition. Then $\u$ satisfies $\u\times\nu\not\equiv \0$ and $\curl\u\cdot\nu\not\equiv0$ on $\p\O$.
    \item[(iv)] Let $\u$ be a first Stokes eigenfunction. Then $\u$ satisfies $\curl\u\times\nu\not\equiv\0$ on $\p\O$.
\end{itemize}
\end{Thm}

This paper is organized as follows. Section 3 contains some notations, definitions and two useful lemmas. In Section 4, we show the proof of Theorem~\ref{ZZ1} and \ref{ZZ2}. In Section 5, we deal with the proof of Theorem~\ref{ZZ3}.

\section{Preliminaries}

We say a domain $D\subset\R^3$ is a star-shaped domain if there exists a point $x_0\in D$ such that for every $x\in \overline{D}$ the line segment $\overline{x_0x}=\{t x_0+(1-t) x, t\in [0,1]\}$ lies in $\overline{D}$. For simplicity, we may assume that $x_0$ is the origin point in the sequel. Note that if the star-shaped domain $D$ is of class $C^1$, we have $x\cdot \nu(x)\geq 0$ for all $x\in \p D$, where $\nu$ denotes the unit outer normal to the boundary. For the proof of such property, we refer to the Lemma (Normals to a star-shaped region) in
\cite[p.515]{Evans}.

Next we give the definitions of the first Maxwell eigenvalues and the first Stokes eigenvalue. Throughout this paper, we always make the following two assumptions on the domain $\O$ when we talk about the first Maxwell eigenvalues.
\begin{enumerate}
\item[$(a)$] $\Omega\subset\mathbb{R}^n$ ($n=2,3$) is a bounded domain with $C^{1,1}$ boundary $\partial\Omega$, and $\Omega$ is locally situated on one side of $\partial\Omega$; $\partial\Omega$ has $m+1$ connected components $\Gamma_0$, $\Gamma_1$, $\cdots$, $\Gamma_m$, where $\Gamma_0$ denotes the boundary
of the infinite connected component of $\mathbb{R}^n\backslash\overline{\Omega}$.
\item[$(b)$] The domain $\Omega$ which can be multiply connected, is made simply connected by $N$ regular cuts $\Sigma_1$, $\Sigma_2$, $\cdots$, $\Sigma_N$ which are of class $C^2$; the $\Sigma_i$, $i=1,2,\cdots,N$ satisfying $\Sigma_i\cap\Sigma_j=\emptyset$ for $i\neq j$ are non-tangential to $\partial\Omega$. Hence the open set $\dot{\Omega}=\Omega\backslash\Sigma$ (with $\Sigma=\cup_{i=1}^N\Sigma_i$) is simply connected and Lipschitz.
\end{enumerate}
We say that $\O$ is simply connected if $N=0$, and $\O$ has no holes if $m=0$.
We denote
$$
\aligned
&\Bbb H_1(\O)=\{\u\in L^2(\O,\Bbb R^n): \curl\u=\0,\; \divg\u=0\text{ in $\O$,}\; \u\cdot\nu=0 \text{ on $\p\O$}\},\\
&\Bbb H_2(\O)=\{\u\in L^2(\O,\Bbb R^n): \curl\u=\0,\;\divg\u=0\text{ in $\O$,}\; \u_T=\0 \text{ on $\p\O$}\},
\endaligned
$$
where $\u_T=\u-(\u\cdot\nu)\nu$. Let $k=n(n-1)/2$, where $n=2,3$. We denote
\begin{equation*}
\aligned
& A=\{\u\in L^2(\O,\mathbb{R}^n)\cap \mathbb{H}_2(\O)^\perp:\divg\u=0 \text{ in $\O$, } \curl\u\in L^2(\O,\mathbb{R}^k),\; \u_T=\0\text{ on $\p\O$}\},\\
& B=\{\u\in L^2(\O,\mathbb{R}^n)\cap \mathbb{H}_1(\O)^\perp:\divg\u=0\text{ in $\O$, }\curl\u\in L^2(\O,\mathbb{R}^k),\; \u\cdot\nu=0\text{ on $\p\O$}\},\\
& C=\{\u\in H_0^1(\O,\mathbb{R}^n):\divg\u=0\text{ in $\O$}\}.
\endaligned
\end{equation*}
The first Maxwell eigenvalue under tangent boundary condition is defined by
$$\alpha_1=\inf_{\0\not\equiv\u\in A}I(\u),$$
and the first Maxwell eigenvalue under normal boundary condition is defined by
$$\beta_1=\inf_{\0\not\equiv\u\in B}I(\u),$$
where
$$I(\u)=\frac{\int_{\O}|\curl\u|^2dx}{\int_{\O}|\u|^2dx}.$$
For some properties related to Maxwell eigenvalue problems, see \cite{CD1999,Yin2012}. The first Stokes eigenvalue is defined by
$$\gamma_1=\inf_{\0\not\equiv\u\in C}\frac{\int_{\O}|\nabla\u|^2dx}{\int_{\O}|\u|^2dx}.$$
For three-dimensional bounded domain, it holds that $\alpha_1=\beta_1$, see \cite[p.440]{Pauly}. Moreover, if the domain is bounded and convex, Pauly \cite[Lemma 4]{Pauly2015} proved that $\alpha_1=\beta_1\geq \mu_2$, where $\mu_2$ is the second eigenvalue of $-\Delta$ under the Neumann boundary condition. For two-dimensional bounded domain, Kelliher\cite[Theorem 1.1]{Kelliher} showed that $\gamma_1>\lambda_1$, where $\lambda_1$ is the first eigenvalue of $-\Delta$ under the Dirichlet boundary condition.

In the last part of this section, we establish two lemmas which play a key role in the proofs of the main theorems. This is inspired by the work of ~\cite[p.180]{Va}.

\begin{Lem}\label{Lem-2.2}
  Let $D$ be a bounded Lipschitz domain in $\R^3$, $\u\in H^1(D,\R^3)$ and $\varphi$ be smooth in $\overline{D}$. Then it holds that
 \begin{equation}\label{eq2.1}
  \aligned
  &\int_D\curl\u\times\u\cdot\d\varphi\,dx+\int_D(\divg\u)(\u\cdot\d\varphi)\,dx \\
=&\int_D\frac{|\u|^2}{2}\Delta\varphi-\sum_{i,j=1}^3u_iu_j\frac{\p^2\varphi}{\p x_i\p x_j}\,dx+\int_{\p D}\frac{|\u|^2}{2}\d\varphi\cdot\nu-(\u\times\d\varphi)\cdot(\u\times\nu)\,dS.
  \endaligned
\end{equation}
\end{Lem}

\begin{proof}
Applying the Green's formula for the operator $\curl$, we find that
\begin{equation}\label{eq-1}
  \aligned
  &\int_D\curl\u\times\u\cdot\d\varphi\,dx=\int_D \u\times\d\varphi\cdot\curl\u\,dx\\
  =&\int_D\u\cdot\curl(\u\times\d\varphi)\,dx+\int_{\p D}\u\times(\u\times\d\varphi)\cdot\nu\,dS\\
  =&\int_D\u\cdot\left[\divg(\d\varphi)\u-(\divg\u)\d\varphi+(\d\varphi\cdot\d)\u-(\u\cdot\d)\d\varphi\right]\,dx\\
  &-\int_{\p D}(\u\times\d\varphi)\cdot(\u\times\nu)\,dS.
  \endaligned
\end{equation}
Here we have used the identity
$$\curl (\mathbf{a}\times \mathbf{b})=(\divg \mathbf{b})\mathbf{a}-(\divg \mathbf{a})\mathbf{b}+(\mathbf{b}\cdot\nabla) \mathbf{a}-(\mathbf{a}\cdot\nabla) \mathbf{b},$$
where $\mathbf{a}$ and $\mathbf{b}$ are two vector fields.

Noting that
$$\aligned
\int_D(\d\varphi\cdot\d)\u\cdot\u\,dx=\int_D\d\varphi\cdot\d\frac{|\u|^2}{2}\,dx=\int_{\p D}\frac{|\u|^2}{2}\d\varphi\cdot\nu\,dS-\int_D\frac{|\u|^2}{2}\Delta\varphi\,dx,
\endaligned
$$
and substituting this into~\eqref{eq-1}, we get
\begin{equation}
  \aligned
  &\int_D\curl\u\times\u\cdot\d\varphi\,dx+\int_D(\divg\u)(\u\cdot\d\varphi)\,dx\\
  =&\int_D\frac{|\u|^2}{2}\Delta\varphi-(\u\cdot\d)\d\varphi\cdot\u\,dx+\int_{\p D}\frac{|\u|^2}{2}\d\varphi\cdot\nu-(\u\times\d\varphi)\cdot(\u\times\nu)\,dS\\
  =&\int_D\frac{|\u|^2}{2}\Delta\varphi-\sum_{i,j=1}^3u_iu_j\frac{\p^2\varphi}{\p x_i\p x_j}\,dx+\int_{\p D}\frac{|\u|^2}{2}\d\varphi\cdot\nu-(\u\times\d\varphi)\cdot(\u\times\nu)\,dS.
  \endaligned
\end{equation}
Thus, \eqref{eq2.1} holds.
\end{proof}

\begin{Lem}\label{one-ball}
Let $D$ be a bounded Lipschitz domain in $\mathbb{R}^3$. Assume that $\0\notin \overline{D}$ and $\u\in H^1(D,\R^3)$. Then for any $\alpha\geq 0$, it holds that
\begin{equation}\label{weight-0}
\aligned
&\int_{D}\curl\mathbf{u}\times\mathbf{u}\cdot \frac{x}{|x|^\alpha}+\left(\mathbf{u}\cdot \frac{x}{|x|^\alpha}\right)\divg \mathbf{u}\,dx=\\
&\int_{D}\frac{1}{|x|^\alpha}\left[\frac{1-\alpha}{2}|\mathbf{u}|^2
+\alpha\left(\mathbf{u}\cdot\frac{x}{|x|}\right)^2\right]\,dx+ \int_{\partial D}\left(\mathbf{u}\cdot \frac{x}{|x|^\alpha}\right)(\mathbf{u}\cdot \nu)-\frac{|\mathbf{u}|^2}{2}\left(\frac{x}{|x|^\alpha}\cdot\nu\right)\,dS=\\
&\int_{D}\frac{1}{|x|^\alpha}\left[\frac{1-\alpha}{2}|\mathbf{u}|^2
+\alpha\left(\mathbf{u}\cdot\frac{x}{|x|}\right)^2\right]dx+\int_{\partial D}\frac{|\mathbf{u}|^2}{2}\left(\frac{x}{|x|^\alpha}\cdot\nu\right)-\left(\mathbf{u}\times \frac{x}{|x|^\alpha}\right)\cdot(\mathbf{u}\times\nu)dS.
\endaligned
\end{equation}
\end{Lem}
 Note that the condition $\0\notin\overline{D}$ is only needed when $\a>0$.

\begin{proof}
The result follows immediately from Lemma~\ref{Lem-2.2} by choosing
$$\varphi=
\begin{cases}
\frac{1}{2-\a}|x|^{2-\a} &\text{if $\a\neq 2$},\\
\ln|x|&\text{if $\a=2$}.
\end{cases}
$$
Indeed, by choosing such $\varphi$, we see that $\d\varphi=x/|x|^\a$. Moreover, we have
\begin{equation}\label{eq-3}
  \Delta\varphi=\divg\left(\frac{x}{|x|^\a}\right)=\frac{3-\a}{|x|^\a},\q \frac{\p^2 \varphi}{\p x_i\p x_j}=\frac{\p }{\p x_i}\left(\frac{x_j}{|x|^\a}\right)=\frac{\delta_{ij}}{|x|^\a} -\frac{\a x_i x_j}{|x|^{\a+2}}\q\text{for } x\neq 0.
\end{equation}
Hence
$$
\int_D\frac{|\u|^2}{2}\Delta\varphi\,dx=\frac{3-\a}{2}\int_D\frac{|\u|^2}{|x|^\a}\,dx
$$
and
$$
\int_\O \sum_{i,j=1}^3u_iu_j\frac{\p^2\varphi}{\p x_i\p x_j}\,dx=\int_D \frac{|\u|^2}{|x|^\a}-\frac{\a}{|x|^\a}\left(\u\cdot\frac{x}{|x|}\right)^2\,dx.
$$
Therefore, the second equality in~\eqref{weight-0} holds. Thanks to the fact that
\begin{equation}\label{Lagrange}
(\mathbf{u}\times x)\cdot(\mathbf{u}\times \nu)=|\mathbf{u}|^2(x\cdot\nu)-(\mathbf{u}\cdot x)(\mathbf{u}\cdot \nu),
\end{equation}
the first equality in~\eqref{weight-0} also holds.
\end{proof}

\section{Uniqueness results for Beltrami flow}

\begin{proof}[Proof of Theorem~\ref{ZZ1}]
Since $\curl\u\times\u=\0$ and $\divg\u=0$ in $\O$ and $\u\times\nu=\0$ on $\p\O$, by setting $\a=0$ and $D=\O$ in the second equality in ~\eqref{weight-0}, we find that
\begin{equation}\label{eq2.3}
0=\int_{\Omega}\frac{|\mathbf{u}|^2}{2}\,dx +\int_{\partial\Omega}\frac{|\mathbf{u}|^2}{2}(x\cdot\nu)\,dS.
\end{equation}
Since $\O$ is star-shaped, we have $x\cdot\nu\geq 0$ on $\p \O$, hence we get $\u\equiv\0$ in $\O$.
\end{proof}

\begin{proof}[Proof of Theorem \ref{ZZ2}]
We first consider case (i). Without loss of generality, we assume $\O$ is a star-shaped domain centered at the origin point.  For fixed $R>0$, we choose $r\in (0,R)$ small enough such that $\overline{B_r(\0)}\subseteq B_R(\mathbf{0})\cap\O$. Set $D=(B_R(\0)\cap \Omega)\backslash \overline{B_r(\0)}$ in Lemma \ref{one-ball}, then we get
 \begin{equation*}
  \begin{aligned}
 &\int_{(B_R(\0)\cap \Omega)\backslash \overline{B_r(\0)}}\frac{1}{|x|^\a}\left(\frac{1-\a}{2}|\u|^2+ \a\left(\u\cdot\frac{x}{|x|}\right)^2\right)dx\\
= &\int_{\partial (B_R(\0)\cap \Omega)\cup\partial B_r(\0)}\left(\mathbf{u}\times \frac{x}{|x|^\alpha}\right)(\mathbf{u}\times\nu)-\frac{|\mathbf{u}|^2}{2}\left(\frac{x}{|x|^\alpha}\cdot\nu\right)dS,
  \end{aligned}
  \end{equation*}
 from the second equality in~\eqref{weight-0}, where $0<\a<3$.
  Letting $r\rightarrow 0^+$, we find that
  \begin{equation}\label{BR-a}
  \begin{aligned}
 &\int_{B_R(\0)\cap \Omega}\frac{1}{|x|^\a}\left(\frac{1-\a}{2}|\u|^2+ \a\left(\u\cdot\frac{x}{|x|}\right)^2\right)dx\\
= &\int_{\partial (B_R(\0)\cap \Omega)}\left(\mathbf{u}\times \frac{x}{|x|^\alpha}\right)(\mathbf{u}\times\nu)-\frac{|\mathbf{u}|^2}{2}\left(\frac{x}{|x|^\alpha}\cdot\nu\right)dS.
  \end{aligned}
  \end{equation}
 Choosing $\a=1$ and using $\u\times\nu=\0$ on $\p\O$, we get
  \begin{equation}\label{BR-O}
  \aligned
 &\int_{B_R(\0)\cap \Omega}\frac{1}{|x|}\left(\u\cdot\frac{x}{|x|}\right)^2\,dx\\
 = &\int_{\p B_R(\0) \cap \Omega}\left|\u\times \frac{x}{|x|}\right|^2-\frac{|\mathbf{u}|^2}{2}dS-\int_{B_R(\0)\cap \p \O}\frac{|\mathbf{u}|^2}{2}\left(\frac{x}{|x|}\cdot\nu\right)dS.
 \endaligned
  \end{equation}
  Since $\Omega$ is star-shaped, we have $x\cdot\nu\geq 0$ on $\p\O$, thus we have
  \begin{equation}\label{tangen}
  \int_{B_R(\0) \cap \Omega}\frac{1}{|x|}\left(\u\cdot\frac{x}{|x|}\right)^2\,dx\leq \int_{\p B_R(\0) \cap \Omega}\left|\u\times \frac{x}{|x|}\right|^2-\frac{|\mathbf{u}|^2}{2}dS,
   \end{equation}
  which yields
  $$\frac{1}{|x|}\left(\u\cdot\frac{x}{|x|}\right)^2\in L^1(\O).$$
  Recall that $|\u(x)|=o(|x|^{-1})$ as $|x|\to +\infty$, thus by letting $R\to +\infty$ in ~\eqref{tangen}, we have
 $$\int_{\Omega}\frac{1}{|x|}\left(\u\cdot\frac{x}{|x|}\right)^2\,dx=0.$$
  Hence $\u\cdot x\equiv0$ in $\O$. Substitute it into \eqref{BR-O}, we get
 \begin{equation}\label{eq-3.4}
 \int_{B_R(\0)\cap \p \O}\frac{|\mathbf{u}|^2}{2}\left(\frac{x}{|x|}\cdot\nu\right)dS=\int_{\p B_R(\0) \cap \Omega}\left|\u\times \frac{x}{|x|}\right|^2-\frac{|\mathbf{u}|^2}{2}dS.
 \end{equation}
 Letting $R\to+\infty$ and using the fact $|\u(x)|=o(|x|^{-1})$ as $|x|\to+\infty$ again, we find that
  $$
 \int_{\p \Omega}\frac{|\mathbf{u}|^2}{2}\left(\frac{x}{|x|}\cdot\nu\right)dS=0.
  $$
  Since $x\cdot\nu\geq 0$ on $\p\O$, then $|\u|^2(x\cdot\nu)\equiv 0$ on $\p\O$. Substituting this into~\eqref{eq-3.4}, we find that
  $$
 \int_{\p B_R(\0) \cap \Omega}\frac{|\mathbf{u}|^2}{2}dS=0,
  $$
 here we have used Lagrange's identity
 $$|\u|^2=\left|\u\times \frac{x}{|x|}\right|^2+\left|\u\cdot \frac{x}{|x|}\right|^2=\left|\u\times \frac{x}{|x|}\right|^2.$$
 Hence we have $\u\equiv \0$ on $\p B_R(\0)\cap \Omega$ for any $R>0$. Substituting these conclusions into
 \eqref{BR-a} with $\a\in (1,3)$, it implies that
 $$\frac{1-\a}{2}\int_{B_R(\0)\cap \Omega}\frac{|\u|^2}{|x|^\a}dx=0\q \text{for all }R>0.
 $$
Thus $\u\equiv\0$ in $\O$.

  \vskip0.05in

  Now we consider case (ii). Using the fact that $\divg\u=0$, $\curl\u\times \u=\0$ in $\O$ and $\u\cdot\nu=0$ on $\p\O$, and by the first equality in Lemma \ref{one-ball} with $\a=1$ and $D=B_R(\0)\cap\O$, we have
  $$\aligned
  \int_{B_R(\0) \cap \Omega}\frac{1}{|x|}\left(\u\cdot\frac{x}{|x|}\right)^2\,dx=&\int_{\p B_R(\0) \cap \Omega}\frac{|\u|^2}{2}-\left(\u\cdot\frac{x}{|x|}\right)^2dS\\
  &+\int_{B_R(\0)\cap \p \Omega} \frac{|\u|^2}{2}\left(\frac{x}{|x|}\cdot\nu\right)\,dS.
  \endaligned$$
   Recall that $\O$ is the complement of the closure of a star-shaped domain, we have $x\cdot\nu\leq 0$ on $\p\O$. Therefore,
  $$
  \int_{B_R(\0) \cap \Omega}\frac{1}{|x|}\left(\u\cdot\frac{x}{|x|}\right)^2\,dx\leq\int_{\p B_R(\0) \cap \Omega}\frac{|\u|^2}{2}-\left(\u\cdot\frac{x}{|x|}\right)^2dS.
  $$
 Then, similar to the proof in case (i), we deduce that $\u\equiv\0$ in $\O$.
\end{proof}

\begin{Rem}
  If we require $\u=\0$ on $\p\O$, then we can drop the assumptions on the shape of the domain in Theorem ~\ref{ZZ2}. In fact, as the proof of Theorem~\ref{ZZ2} shows, we can get~\eqref{BR-a} without using the boundary condition and the assumption of ``star-shaped". Therefore, starting with ~\eqref{BR-a} and using $\u=\0$ on $\p\O$, we can see that
  $$
  \int_{B_R(\0)\cap \Omega}\frac{1}{|x|}\left(\u\cdot\frac{x}{|x|}\right)^2\,dx= \int_{\p B_R(\0) \cap \Omega}\left|\u\times \frac{x}{|x|}\right|^2-\frac{|\mathbf{u}|^2}{2}dS.
  $$
  Then, similar to the arguments in the lines after~\eqref{tangen}, we finally get $\u\equiv\0$ in $\O$.
\end{Rem}

\section{Maxwell and Stokes eigenvalue problems}

\begin{proof}[Proof of Theorem \ref{ZZ3}]
First, we claim that a bounded and star-shaped $C^{1,1}$ domain must be simply connected and have no holes. From the definition of the star-shaped domain, it is easy to find that any closed curve is homotopic to the center point of the star-shaped domain. So $\O$ is simply-connected. Hence~$\mathbb{H}_1(\O)=\{\0\}$. On the other hand, $\mathbb{H}_2(\O)=\{\0\}$. In fact, for any $\h\in\mathbb{H}_2(\O)$, $\h$ satisfies the following equations
\begin{equation*}
\begin {cases}
\curl\mathbf{h}\times\mathbf{h}=\mathbf{0}& \text{ \rm in } \Omega,\\
\divg\mathbf{h}=0 &\text{ \rm in } \Omega,\\
\mathbf{h}\times\nu=\mathbf{0} &\text{ \rm on } \partial\Omega.\\
\end {cases}
\end{equation*}
By Theorem \ref{ZZ1}, we have $\h=\0$. Thus $\O$ has no holes.

By \cite[p.209, Theorem 3]{DL1990} and the fact that the embedding $H^1(\O,\mathbb{R}^3)\hookrightarrow L^2(\O,\mathbb{R}^3)$ is compact, $\alpha_1,\beta_1,\gamma_1$ can be attained. Thanks to Poincar\'{e} type inequalities, $\alpha_1,\beta_1,\gamma_1$ are positive. In order to prove $\alpha_1=\beta_1<\gamma_1$, first we show~$\alpha_1=\beta_1\leq \gamma_1$, and then we show~$\alpha_1=\beta_1\neq\gamma_1$.
Since
$$\gamma_1=\inf_{\0\not\equiv\u\in C}\frac{\int_{\O}|\nabla\u|^2dx}{\int_{\O}|\u|^2dx}=\inf_{\0\not\equiv\u\in C}\frac{\int_{\O}|\curl\u|^2dx}{\int_{\O}|\u|^2dx},$$
and $C=A\cap B$, it is easy to find~$\alpha_1=\beta_1\leq \gamma_1$. We prove~$\alpha_1=\beta_1\neq\gamma_1$ by contradiction. We assume that~$\alpha_1=\beta_1=\gamma_1$ holds. Since~$\gamma_1$~can be attained, there exists~$\0\not\equiv\u\in C$~such that~$I(\u)=\gamma_1$. Hence~$I(\u)=\beta_1$ and~$\u$~satisfies the following Euler-Lagrange equations
\begin{equation*}
\begin {cases}
\curl\curl \u=\beta_1\u & \text{ \rm in }\O,\\
\divg\u=0 &\text{ \rm in }  \O,\\
\u\cdot\nu=0 &\text{ \rm on }  \p\O,\\
\curl\u\times\nu=\0 &\text{ \rm on }  \p\O.\\
\end {cases}
\end{equation*}
Since~$\divg(\curl\u)=0$ in $\O$, $\curl(\curl\u)=\beta_1\u$ in $\O$, $\curl\u\times\nu=\0$ on $ \p\O$,
by the regularity theory for the $\divg$-$\curl$ system we get~$\curl\u\in H^1(\O,\mathbb{R}^3)$. Applying Lemma \ref{one-ball}, we obtain the following two equalities£º
$$\int_{\O}\curl\u\times\u\cdot xdx=\int_{\O}\frac{|\u|^2}{2}dx-\int_{\p\O}\frac{|\u|^2}{2}(x\cdot\nu)dS,$$
$$\int_{\O}\curl(\curl\u)\times\curl\u\cdot xdx=\int_{\O}\frac{|\curl\u|^2}{2}dx+\int_{\p\O}\frac{|\curl\u|^2}{2}(x\cdot\nu)dS.$$
Combining the above two equalities with the following equality
$$\int_{\O}\curl(\curl\u)\times\curl\u\cdot xdx=-\beta_1\int_{\O}\curl\u\times\u\cdot xdx,$$
we have
\begin{equation}\label{eigen-beta}
\int_{\O}|\curl\u|^2+\beta_1|\u|^2dx+\int_{\p\O}|\curl\u|^2(x\cdot\nu)dS=\beta_1\int_{\p\O}|\u|^2(x\cdot\nu)dS.
\end{equation}
Since~$\u=\0$ on $\p\O$ and $x\cdot\nu\geq0$ on $\p\O$, the above equality implies~$\u\equiv\0$, which is a contradiction.

Next we prove the properties of the first Maxwell eigenfunctions and the first Stokes eigenfunctions. Let $\u$ be a first Maxwell eigenfunction under tangent boundary condition. Assume $\u\cdot\nu=0$ on $\p\O$. Then~$\u=\0$ on $\p\O$. It follows~$\gamma_1\leq\alpha_1$, which is a contradiction. Hence~$\u\cdot\nu\not\equiv 0$ on $\p\O$. We can derive that~$\u$~satisfies the Euler-Lagrange equations
\begin{equation*}
\begin {cases}
\curl\curl \u=\alpha_1\u & \text{ \rm in }\O,\\
\divg\u=0 &\text{ \rm in }  \O,\\
\u\times\nu=\0 &\text{ \rm on }  \p\O.\\
\end {cases}
\end{equation*}
Since~$\divg(\curl\u)=0$ in $\O$, $\curl(\curl\u)=\alpha_1\u$ in $\O$, $\curl\u\cdot\nu=0$ on $ \p\O$,
by the regularity theory for the $\divg$-$\curl$ system we get~$\curl\u\in H^1(\O,\mathbb{R}^3)$. Applying Lemma \ref{one-ball}, we have the following two equalities:
$$\int_{\O}\curl\u\times\u\cdot xdx=\int_{\O}\frac{|\u|^2}{2}dx+\int_{\p\O}\frac{|\u|^2}{2}(x\cdot\nu)dS,$$
$$\int_{\O}\curl(\curl\u)\times\curl\u\cdot xdx=\int_{\O}\frac{|\curl\u|^2}{2}dx-\int_{\p\O}\frac{|\curl\u|^2}{2}(x\cdot\nu)dS.$$
From the above two equalities and the following equality
$$\int_{\O}\curl(\curl\u)\times\curl\u\cdot xdx=-\alpha_1\int_{\O}\curl\u\times\u\cdot xdx,$$
we obtain
$$\int_{\O}|\curl\u|^2+\alpha_1|\u|^2dx+\alpha_1\int_{\p\O}|\u|^2(x\cdot\nu)dS=\int_{\p\O}|\curl\u|^2(x\cdot\nu)dS.$$
From the above equality we see~$\curl\u\times\nu\not\equiv \0$ on $\p\O$.

Let $\u$ be a first Maxwell eigenfunction under normal boundary condition. From \eqref{eigen-beta}, we find $\u\times\nu\not\equiv \0$ on $\p\O$. Since $\curl\u$ is a first Maxwell eigenfunction under tangent boundary condition, it follows that $\curl\u\cdot\nu\not\equiv 0$ on $\p\O$.

Finally, we show that any first Stokes eigenfunction~$\u$ satisfies~$\curl\u\times\nu\not\equiv\0$ on $\p\O$. We can derive that~$\u$~satisfies the Euler-Lagrange equations
\begin{equation*}
\begin {cases}
-\Delta\u+\nabla \pi=\gamma_1\u & \text{ \rm in }\O,\\
\divg\u=0 &\text{ \rm in }  \O,\\
\u=\0 &\text{ \rm on }  \p\O.\\
\end {cases}
\end{equation*}
By the regularity theory for the Stokes equations, we have~$(\u,\pi)\in H^2(\O,\mathbb{R}^3)\times H^1(\O)$. Assume $\curl\u\times\nu=\0$ on $\p\O$. Since~$\curl\curl\u=-\Delta\u+\nabla(\divg\u)=\gamma_1\u-\nabla \pi$, we have
$$\frac{\p\pi}{\p\nu}=\gamma_1\u\cdot\nu-\curl\curl\u\cdot\nu=0\text{ on $\p\O$.}$$
Hence~$\pi$~satisfies the equation
\begin{equation*}
\begin {cases}
-\Delta\pi=0 & \text{ \rm in }\O,\\
\,\,\frac{\p\pi}{\p\nu}=0 &\text{ \rm on }  \p\O,\\
\end {cases}
\end{equation*}
Obviously, $\pi$ must be a constant. Thus $\u$~also satisfies the equations
\begin{equation*}
\begin {cases}
\curl\curl \u=\gamma_1\u & \text{ \rm in }\O,\\
\divg\u=0 &\text{ \rm in }  \O,\\
\u\cdot\nu=0 &\text{ \rm on }  \p\O,\\
\curl\u\times\nu=\0 &\text{ \rm on }  \p\O.\\
\end {cases}
\end{equation*}
Hence we can establish an equality similar to \eqref{eigen-beta}. The rest of proof is similar to the proof of (ii).
\end{proof}

\begin{Rem}
For three-dimensional bounded domain, we can give a new proof of $\alpha_1=\beta_1$. Indeed, first we prove $\beta_1\leq\alpha_1$. We assume $\0\not\equiv\u\in A$ attains the minimum with $I(\u)=\alpha_1$. Then $\u$ satisfies the following Euler-Lagrange equations
\begin{equation*}
\begin {cases}
\curl\curl\u=\alpha_1\u & \text{ \rm in } \Omega,\\
\divg\u=0 &\text{ \rm in } \Omega,\\
\u\times\nu=\0 &\text{ \rm on } \partial\Omega.
\end {cases}
\end{equation*}
By the regularity theory for the $\divg$-$\curl$ system, we get $\curl\u \in H^1(\Omega,\mathbb{R}^3)$. Since $\u\times\nu=\0$ on $\p\O$, we have $\0\not\equiv\curl\u\in B$. Hence it holds that
$$\beta_1\leq \frac{\int_{\Omega}|\curl \curl\u|^{2}dx}{\int_{\Omega}|\curl\u|^{2}dx}=
\frac{\int_{\Omega}|\alpha_1\u|^{2}dx}{\alpha_1\int_{\Omega}|\u|^{2}dx}=\alpha_1.$$
Similarly, we obtain $\alpha_1\leq\beta_1$. Thus $\alpha_1=\beta_1$.

\end{Rem}

\begin{Rem}
Let~$\O$~be a bounded $C^{1,1}$ domain in $\mathbb{R}^2$. Then we have $\alpha_1=\mu_2$, $\beta_1=\lambda_1$.
Hence $\alpha_1<\beta_1<\gamma_1$.
\end{Rem}
\begin{proof}
Since~$\alpha_1$~can be attained, there exists~$\0\not\equiv\u\in A$~such that~$I(\u)=\alpha_1$. Hence~$\u$~satisfies the following Euler-Lagrange equations
\begin{equation*}
\begin {cases}
\Curl\curl \u=\alpha_1\u & \text{ \rm in }\O,\\
\divg\u=0 &\text{ \rm in }  \O,\\
\u\times\nu=0 &\text{ \rm on }  \p\O,\\
\end {cases}
\end{equation*}
here the operator $\Curl$ is defined by
$$\Curl w=(\frac{\p w}{\p x_2},-\frac{\p w}{\p x_1})\text{ for all $w\in \mathcal{D}'(\O)$.}$$
By Corollary 5' in \cite[p.224]{DL1990}, there exist $p\in H^1_0(\O)$, $\mathbf{h}\in \mathbb{H}_2(\O)$ and $\varphi\in H^1(\O)\cap \mathbb{R}^\perp$ such that $\u=\d p+\mathbf{h}+\Curl\varphi$. Since $\divg \u=0$ in $\O$ and $\u,\Curl\varphi\in\mathbb{H}_2(\O)^\perp$, we see that $p=0$ and $\mathbf{h}=\0$. Hence $\Curl(\curl\u-\alpha_1\varphi)=\Curl\curl \u-\alpha_1\u=\0$. Thus there exists a constant $C$ such that $\curl\u-\alpha_1\varphi=C$. Integrating this equality over $\O$, we get $C=0$. Consequently, $\varphi\in H^1(\O)\cap \mathbb{R}^\perp$ satisfies
\begin{equation*}
\begin {cases}
-\Delta\varphi=\alpha_1\varphi & \text{ \rm in }\O,\\
\frac{\p\varphi}{\p\nu}=\u\times\nu=0&\text{ \rm on }  \p\O.\\
\end {cases}
\end{equation*}
This means that $\varphi$ must be an eigenfunction of $-\Delta$ under Neumann boundary condition. Therefore $\alpha_1\geq \mu_2$. On the other hand, since $\0\not\equiv\Curl\varphi\in A$, where $\varphi\in H^1(\O)\cap \mathbb{R}^\perp$ is a solution of the following equation
\begin{equation*}
\begin {cases}
-\Delta\varphi=\mu_2\varphi & \text{ \rm in }\O,\\
\frac{\p\varphi}{\p\nu}=0&\text{ \rm on }  \p\O,\\
\end {cases}
\end{equation*}
we obtain
$$
\aligned
\alpha_1\leq\frac{\int_{\O}|\curl\Curl\varphi|^2dx}{\int_{\O}|\Curl\varphi|^2dx}=
\frac{\int_{\O}|-\Delta\varphi|^2dx}{\int_{\O}|\d\varphi|^2dx}
=\frac{\int_{\O}|\mu_2\varphi|^2dx}{\mu_2\int_{\O}|\varphi|^2dx}=\mu_2.
\endaligned
$$
So $\alpha_1=\mu_2$.

Next we show that $\beta_1=\lambda_1$. Since~$\beta_1$~can be attained, there exists~$\0\not\equiv\u\in B$~such that~$I(\u)=\beta_1$. Hence~$\u$~satisfies the following Euler-Lagrange equations
\begin{equation*}
\begin {cases}
\Curl\curl \u=\beta_1\u & \text{ \rm in }\O,\\
\divg\u=0 &\text{ \rm in }  \O,\\
\u\cdot\nu=0 &\text{ \rm on }  \p\O,\\
\curl\u=0 &\text{ \rm on }  \p\O.\\
\end {cases}
\end{equation*}
Indeed, by a standard argument, for all $\mathbf{v}\in H^1(\O,\mathbb{R}^2)$ satisfying $\divg \mathbf{v}=0$ in $\O$ and $\mathbf{v}\cdot\nu=0$ on $\p\O$, we have
$$\int_{\p\O}\curl\u(\mathbf{v}\times\nu)\,dS=0.$$
For any $\psi\in H^{1/2}(\p\O)$, let $\mathbf{v}\in H^1(\O,\mathbb{R}^2)$ solve the following Stokes system
\begin{equation*}
\begin {cases}
-\Delta\mathbf{v}+\d \pi=\0 & \text{ \rm in }\O,\\
\divg\mathbf{v}=0 &\text{ \rm in }  \O,\\
\mathbf{v}=\psi(\nu_2,-\nu_1) &\text{ \rm on }  \p\O.\\
\end {cases}
\end{equation*}
Hence we get
$$\int_{\p\O}\curl\u\,\psi\,dS=0\text{ for any $\psi\in H^{1/2}(\p\O)$}.$$
By Corollary 5 in \cite[p.224]{DL1990}, there exist $p\in H^1(\O)$, $\mathbf{h}\in \mathbb{H}_1(\O)$ and $\varphi\in H^1_0(\O)$ such that $\u=\d p+\mathbf{h}+\Curl\varphi$. Since $\divg \u=0$ in $\O$, $\u\cdot\nu=0$ on $\p\O$ and $\u,\Curl\varphi\in\mathbb{H}_1(\O)^\perp$, we see that $p$ is a constant and $\mathbf{h}=\0$. Hence $\Curl(\curl\u-\beta_1\varphi)=\Curl\curl \u-\beta_1\u=\0$. Thus there exists a constant $C$ such that $\curl\u-\alpha_1\varphi=C$. Since $\curl\u=\varphi=0$ on $\p\O$, we get $C=0$. Consequently, $\varphi\in H_0^1(\O)$ satisfies
\begin{equation*}
\begin {cases}
-\Delta\varphi=\beta_1\varphi & \text{ \rm in }\O,\\
\varphi=0&\text{ \rm on }  \p\O.\\
\end {cases}
\end{equation*}
Similarly, we obtain $\beta_1=\lambda_1$.

\end{proof}

\subsection*{Acknowledgements.}
The authors are grateful to their supervisor, Prof. Xingbin Pan, for guidance and constant encouragement. They would like to thank Prof. Hairong Yuan for bringing them the reference \cite{EPS16}. This work was partially supported by the National Natural Science Foundation of China grant no.11671143. Z. B. Zhang was also partly supported by Outstanding Doctoral Dissertation Cultivation Plan of Action (PY2015038).

\end {document}